\documentclass[10pt]{article}
\usepackage{amsmath}
\usepackage{amssymb}
\usepackage{amsfonts}
\usepackage{amscd}
\usepackage[usenames]{color}
\usepackage{amsthm}

\usepackage[hidelinks,draft=false]{hyperref}
\usepackage{bookmark}

\bookmarksetup{
  open,
  numbered,
  addtohook={%
    \ifnum\bookmarkget{level}>1 
      \renewcommand*{\numberline}[1]{}%
    \fi
  },
}

\usepackage[top=1.22in, bottom=1.47in, left=1.68in, right=1.68in]{geometry} 

\usepackage{bm}

\usepackage{chngcntr}
\counterwithin*{equation}{section}

\input{xypic}
\xyoption{all}

\newtheorem{theorem}{Theorem}[section]
\newtheorem{prop}[theorem]{Proposition}

\newtheorem{lemma}[theorem]{Lemma}

\newtheorem{defn}[theorem]{Definition}

\newtheorem{remark}[theorem]{Remark}
\newenvironment{rem}{\begin{remark}\rm}{\end{remark}}
\newtheorem{example}[theorem]{Example}

\newtheorem{exercise}[theorem]{Exercise}

\def\a{\alpha}
\def\b{\beta}
\def\ga{\gamma}
\def\de{\delta}
\def\ep{\epsilon}
\def\la{\lambda}
\def\si{\sigma}

\def\P{{\mathbb P}}

\def\C{{\mathbb C}}
\def\Z{{\mathbb Z}}

\def\F{{\cal F}}

\def\O{{\cal O}}

\def\CZ{{\cal Z}}

\def\bd{\partial}

\def\dim{\mbox{\rm dim\,}}

\def\Aut{\mbox{Aut}}

\def\Spacing{{\displaystyle\phantom{\bigoplus_A}}}

\def\st{{\scriptscriptstyle \times}}

\def\arr{\xrightarrow{\hspace*{0.4cm}}}

\def\OP{{\rm OP}}

\def\ba{{\bm \alpha}}
\def\bb{{\bm \beta}}
\def\bd{{\bm \delta}}
\def\be{{\bm \eta}}

\def\Cliff{{\textsf{Cliff}}}
\def\SS{{\textsf{S}}}
\def\SC{{\textsf{C}}}
\def\SB{{\textsf{B}}}

\def\SH{{\textsf{SH}}}

\def\f{{\bf f}}

\def\lb{\langle}
\def\rb{\rangle}

\def\SP{{\rm SP}}

\def\ba{{\bm \alpha}}
\def\bb{{\bm \beta}}
\def\bd{{\bm \delta}}
\def\be{{\bm \eta}}
\def\1{{\bm 1}}

\def\vt{\vartheta}

\def\Cob{{\bf Cob}}
\def\Vect{{\bf Vect}}

\title{{\bf A note on Gunningham's formula}}
\author{Junho Lee\thanks{The author was partially supported by  NSF grant  DMS-1206192.} }

\date{\empty}
\addtocounter{section}{0}
\begin{document}

\maketitle

\begin{abstract}
\medskip
Gunningham \cite{G} constructed an extended topological quantum field theory (TQFT)
to obtain  a closed formula for all spin Hurwitz numbers.
In this note, we use the gluing theorem in \cite{LP2}
to re-prove Gunningham's formula. We also describe a TQFT formalism naturally induced  from the gluing theorem.
\end{abstract}

\vskip.15in

\setcounter{section}{0}

\section{Introduction}

Let $X$ be a surface of general type with a smooth canonical divisor $D$.
The complex curve $D$ has genus $h=K_X^2+1$ and
the normal bundle $N$ to $D$ is a theta characteristic
on $D$ (that is, $N^2=K_D$) with $p\equiv h^0(N) \equiv \chi(\O_X)$ (mod 2).
The pair $(D,N)$ is called a {\em spin curve} of genus $h$ with parity $p$.
The Gromov-Witten invariants of $X$ are the same as the local GW invariants of the spin curve $(D,N)$
that depend only on $(h,p)$.
In particular,  for $d>0$ the dimension zero local GW invariant of the spin curve $(D,N)$ is given by the formula
\begin{equation}\label{dim=0GW}
GT_d^{h,p} = \sum_f \frac{(-1)^{h^0(f^*N)}}{|\Aut(f)|},
\end{equation}
where the sum is over all degree $d$ etale covers $f$ (see \cite{LP1,KL,MP}).
One can calculate these local invariants by extending  the (weighted) signed sum to certain ramified covers, which are the spin Hurwitz numbers.

A partition $\a\vdash d$  is odd if all parts in $\a$ are  odd. We set
$$
\OP(d) = \{\,\a\vdash d\,:\, \a\  \text{is odd}\,\}.
$$
Fix $k$ points $y^1,\cdots,y^k$ in $D$ and
consider
degree $d$  holomorphic maps $f:C\to D$ from possibly disconnected curves $C$ of Euler characteristic $\chi(C)$
that are ramified only over the fixed points $y^i$ with ramification profile
$\a^i=(\a^i_1,\cdots,\a^i_{\ell_i})\in \OP(d)$.
By the Riemann-Hurwitz formula, the (ramified) covers $f$ satisfy
\begin{equation}\label{RH}
d\chi(D) - \chi(C)+ \sum_{i=1}^k(\ell(\a^i)-d) = 0,
\end{equation}
where $\chi(D)=2-2h$ and $\ell(\a^i)$ is the length of the partition $\a^i$.
By the Hurwitz formula,
the twisted line bundle
\begin{equation}\label{TwistedTheta}
N_f = f^*N\otimes \O_C\big(\,\sum_{i,j} \tfrac12 (\a^i_j-1) x^i_j\,\big)
\end{equation}
is a theta characteristic on $C$ where
$f^{-1}(y^i)=\{x^i_j\}$ and $f$ has  multiplicity $\a^i_j$ at $x^i_j$.
We define the parity $p(f)$ of a map $f$ as
$$
p(f) \equiv h^0(N_f)\ \ \ (\text{mod 2}).
$$
Given $\a^1,\cdots,\a^k\in \OP(d)$, the spin Hurwitz number of genus $h$ and parity $p$ is defined as a (weighted) sum of  covers $f$ satisfying
(\ref{RH}) with sign determined by  parity $p(f)$:
\begin{equation}\label{def-spinHurwitz}
H^{h,\pm}_{\a^1,\cdots,\a^k} = \sum_f \frac{(-1)^{p(f)}}{|\Aut(f)|}
\end{equation}
where $+$ or $-$ denotes the parity of the spin curve $(D,N)$.
If $k=0$ (or unramified) then this is the etale spin Hurwitz number that equals
the local invariant (\ref{dim=0GW}).
We will call $\chi(C)$ in (\ref{RH}) the {\em domain Euler characteristic} for the spin Hurwitz number
(\ref{def-spinHurwitz}).

\medskip
Eskin, Okounkov and Pandharipande \cite{EOP} first studied the spin Hurwitz numbers for genus $h=1$
with trivial theta characteristic, that is, $(h,p)=(1,-)$.
They related the parity of maps with  combinatorics of the Sergeev group $\SC(d)$.
A partition $\la=(\la_1,\cdots,\la_\ell)$ of $d$ is strict if $\la_1>\cdots>\la_\ell$.
Let $\SP(d)$ denote the set of strict partitions of $d$ and set
$$
\SP^+(d) = \{\,\la\in \SP(d)\,:\, \ell(\la)\ \text{is even}\,\}
\ \ \ \text{and}\ \ \
\SP^-(d) = \{\,\la\in \SP(d)\,:\, \ell(\la)\ \text{is odd}\,\}.
$$
The irreducible spin $\SC(d)$-supermodules  $V^\la$ are indexed by strict partitions $\la \in \SP(d)$ and  the conjugacy class corresponding to an odd partition $\a^i\in \OP(n)$ acts in $V^\la$
as multiplication by a constant, which is the central character $\f_{\a^i}(\la)$
(see Section 1).

\begin{theorem}[\cite{EOP}]
\label{T:EOP}
With the notation as above,
\begin{equation}\label{EOP}
H^{1,-}_{\a^1,\cdots,\a^k} =
2^{\frac{\chi(C)}{2}}\Big(
\sum_{\la\in \SP^+(d)} \prod_i \f_{\a^i}(\la) \ \ -
\sum_{\la\in \SP^-(d)} \prod_i \f_{\a^i}(\la)
\ \Big),
\end{equation}
where $\chi(C)$ is the domain Euler characteristic.
\end{theorem}

Recently, Gunningham \cite{G} constructed a fully extended (spin) topological quantum
field theory (TQFT). His extended TQFT gives a formula for all spin Hurwitz numbers.
For each strict partition $\la\in\SP(d)$, let $V^\la$ be as above and  set
\begin{equation}\label{dim}
c_\la = \frac{\dim V^\la}{|\SC(d)|} 
\end{equation}

\begin{theorem}[\cite{G}]
\label{T:G}
\begin{equation}\label{G}
H^{h,\pm}_{\a^1,\cdots,\a^k} =
2^{\frac{\chi(C)+\chi(D)}{2}}
\Big(
\sum_{\la\in \SP^+(d)} 2^{\frac{\chi(D)}{2}}c_\la^{\chi(D)} \prod_i \f_{\a^i}(\la)
\ \pm\sum_{\la\in \SP^-(d)} c_\la^{\chi(D)}\prod_i \f_{\a^i}(\la)
\Big),
\end{equation}
where $\chi(D)=2-2h$ and $\chi(C)$ is the domain Euler characteristic.
\end{theorem}

Independently, Parker and the author \cite{LP2} adapted
 the degeneration method of the GW theory to obtain a gluing theorem for spin Hurwitz numbers.
For a partition $\ga\vdash d$, let $\ga(k)$ be the number of parts of size $k$ in $\ga$ and set
\begin{equation}\label{order-z}
z_\ga = \prod_k k^{\ga(k)} \ga(k)!
\end{equation}

\begin{theorem}[\cite{LP2}]
\label{T:LP2}
Let $\a^1,\cdots,\a^s,\b^1,\cdots,\b^r\in \OP(d)$. We have
\begin{align*}
H^{h,p}_{\a^1,\cdots,\a^s,\b^1,\cdots,\b^r} &=
\sum_{\ga\in \OP(d)} z_\ga\cdot
H^{h_1,p_1}_{\a^1,\cdots,\a^s,\ga}\cdot H^{h_2,p_2}_{\b^1,\cdots,\b^r,\ga},
\\
H^{h+1,p}_{\a^1,\cdots,\a^s} &=
\sum_{\ga\in\OP(d)} z_\ga\cdot H^{h,p}_{\a^1,\cdots,\a^s,\ga,\ga},
\end{align*}
where $h=h_1+h_2$ and $p\equiv p_1+p_2$ (mod 2).
\end{theorem}

Our main goal is to reprove Gunningham's formula (\ref{G}).
To that end, we need to calculate the $(h,p)=(0,+)$ spin Hurwitz numbers.

Section~1 gives a brief review of the representation theory of the Sergeev group $\SC(d)$ and
a key fact (Lemma~\ref{key}) about the central characters of $\SC(d)$.

Section~2 follows the approach of \cite{EOP} to show:
\begin{equation}\label{h=0}
H^{0,+}_{\a^1,\cdots,\a^k} =
2^{\frac{\chi(C)+2}{2}}
\Big(
\sum_{\la\in \SP^+(d)} 2c_\la^{2} \prod_i \f_{\a^i}(\la)
\ +\sum_{\la\in \SP^-(d)} c_\la^{2}\prod_i \f_{\a^i}(\la)
\Big).
\end{equation}

In Section 3, we use the gluing theorem  with (\ref{EOP}) and (\ref{h=0}) to prove
the formula (\ref{G}).
We also observe that  the  formula (\ref{G}) gives the gluing theorem (see Remark~\ref{GGluing}).

The spin Hurwitz numbers are not defined for $(h,p)=(0,-)$ since the only theta characteristic on $\P^1$ is $\O(-1)$.
In Section 4, we first extend the gluing theorem to include the case $(h,p)=(0,-)$ and then describe  a TQFT formalism naturally induced from the (extended) gluing theorem.

\section{Representations of the Sergeev group}

\addtocounter{section}{0}

This section  reviews the representation theory of the Sergeev group  relevant to our discussion.
We generally follow the notation and terminology of \cite{EOP}.
For proofs and more details, we refer to \cite{J1,J2,Sg} and  \cite[Ch.3]{CW} and references therein.

\subsection{Sergeev group}

The Sergeev group $\SC(d)$ is the semidirect product
$$
\SC(d) = \Cliff(d) \rtimes \SS(d),
$$
where $\Cliff(d)$ is the Clifford group generated by  $\xi_1,\cdots,\xi_d$ and a
central element $\ep$ subject to the relations
$$
\xi_i^2 = 1,\ \ \ \ \ep^2=1,\ \ \ \ \xi_i\xi_j = \ep\,\xi_j\xi_i \ \ (i\ne j),
$$
and the symmetric group $\SS(d)$ on $d$ letters acts on $\Cliff(d)$ by
permuting the $\xi_i$'s.

The group $\SC(d)$ is a double cover of the hyperoctahedral group $\SB(d)=\Z_2^d \rtimes \SS(d)$.
Since $\Cliff(d)/\{1,\ep\} \cong \Z_2^d$, setting $\ep=1$ gives
a short exact sequence of groups
\begin{equation*}\label{central}
0\ \arr\ \Z_2\ \arr\ \SC(d)\ \overset{\theta}{\arr}\ \SB(d)\ \arr\ 0.
\end{equation*}
The group $\SB(d)$ embeds in
the symmetric group $\SS(2d)$ on the set $\{\pm 1, \cdots,\pm d\}$ via
$$
\xi_ig(\pm k) =
\left\{
\begin{array}{rr}
\mp g(k)\ \ &\text{if}\ g(k)=i \\
\pm g(k)\ \ &\text{if}\ g(k)\ne i
\end{array}\right.
$$
Notice that $\SB(d)$ is the centralizer of the involution $k\to -k$ in $\SS(2d)$.

\subsection{Conjugacy classes}
\label{conj}

The symmetric group $\SS(d)$ embeds in $\SB(d)$ and $\SC(d)$.
An element $g$ of $\SB(d)$ and $\SC(d)$ is a {\em pure permutation} if $g\in \SS(d)$.
Define a $\Z_2$-grading on $\SC(d)$  by setting
\begin{equation}\label{grading}
\deg \xi_i = 1,\ \ \ \deg(g) = \deg(\ep) = 0,\ \ \ g\in\SS(d).
\end{equation}
An even (resp. odd) conjugacy class is a conjugacy class of an even (resp. odd) element.
Observe that for each conjugacy class $C$, either $C\cap \ep C=0$ or $C=\ep C$.
\begin{itemize}
\item[(a)]
Let $C_\ga$ be the conjugacy class in $\SC(d)$ of a pure permutation $g$  of cycle type $\ga\in \OP(d)$. Then, $C_\ga$ and $\ep C_\ga$ are disjoint even conjugacy classes and
\begin{equation}\label{num-conj}
|C_\ga| = |\ep C_\ga| = \frac{|\SC(d)|}{2^{\ell(\ga)+1}z_\ga}
\end{equation}
where $z_\ga$ is the order of the centralizer of $g$ in $\SS(d)$ given by  (\ref{order-z}).
\item[(b)]
We can write all conjugacy classes of $\SC(d)$ as
\begin{equation}\label{all-conj}
\underbrace{C_1,\ep C_1,\cdots,C_m,\ep C_m}_{\text{even}}\,,\,
\underbrace{\hat{C}_1,\ep\hat{C}_1,\cdots,\hat{C}_q,\ep\hat{C}_q}_{\text{odd}}\,,\,
\underbrace{\tilde{C}_1,\cdots,\tilde{C}_s}_{\ep\,\tilde{C}_i=\tilde{C}_i},
\end{equation}
where $m=|\OP(d)|=|\SP(d)|$ and $q=|\SP^-(d)|$.
\end{itemize}
The denominator $2^{\ell(\ga)+1}z_\ga$ in (\ref{num-conj}) is the order of the centralizer of $g$ in $\SC(d)$.

\begin{defn}
For a partition $\ga\vdash d$, we define
$$
\vt_\ga = 2^{\ell(\ga)+1}z_\ga.
$$
\end{defn}

\subsection{Spin {\rm $\SC(d)$}-supermodules}

For a finite group $G$, let $G^\wedge$ denote the set of irreducible complex representations of $G$.
The central element $\ep$ acts as multiplication by either $+1$ or $-1$ on each $V\in \SC(d)^\wedge$. If $\ep$ acts on $V$ as multiplication by $1$, then $V\in \SB(d)^\wedge$.
Let $\SC(d)^\wedge_-$ be the set of irreducible complex representations of $\SC(d)$ on which
$\ep$ acts as multiplication by $-1$. We have
\begin{equation}\label{decomp-irr}
\SC(d)^\wedge = \SB(d)^\wedge \,\cup\, \SC(d)^\wedge_-.
\end{equation}

The grading (\ref{grading}) makes the group algebra $\C[\SC(d)]$  a semisimple associative superalgebra. A {\em spin $\SC(d)$-supermodule} is a supermodule  over $\C[\SC(d)]$ on which  $\ep$ acts as multiplication by \nolinebreak $-1$.
The irreducible (or simple) spin $\SC(d)$-supermodules are indexed by strict partitions $\la\in \SP(d)$.
For each $\la\in\SP(d)$, let $V^\la$ be its corresponding irreducible spin $\SC(d)$-supermodule.
\begin{itemize}
\item[(c)]
For $\la\in \SP^+(d)$, we have $V^\la\in \SC(d)^\wedge_-$.
\item[(d)]
For $\la\in \SP^-(d)$, we have $V^\la=V^\la_0\oplus V^\la_1$ (as a module over $\C[\SC(d)]$) such that $V^\la_0, V^\la_1\in \SC(d)^\wedge_-$ and they are not isomorphic.
\end{itemize}

\subsection{Central characters}

For $\la\in\SP(d)$, let $\zeta^\la$ denote the character of the irreducible $\C(d)$-supermodule $V^\la$.
By (\ref{all-conj}), the character $\zeta^\la$ is determined by its values
$\zeta^\la(C_\ga)=-\zeta^\la(\ep C_\ga)$  on
even conjugacy classes $C_\ga$ and  $\ep C_\ga$ where $\ga\in\OP(d)$.
For $\la,\mu\in \SP(d)$,
\begin{equation}\label{inner-prod1}
\lb \zeta^\la,\zeta^\mu\rb =
\sum_{\ga\in \OP(d)} \frac{2}{\vt_\ga} \zeta^\la(C_\ga) \zeta^\mu(C_\ga)  =
\left\{
\begin{array}{cl}
\delta_{\la\mu}\ \ &\text{if}\ \ \la\in \SP^+(d) \\
2\delta_{\la\mu}\ \ &\text{if}\  \ \la\in \SP^-(d)
\end{array}\right.
\end{equation}
where $\lb\cdot ,\cdot \rb$ is the inner product on the space of class functions of  the finite group $\SC(d)$.

For each $\ga\in \OP(d)$,  the class sum $\overline{C}_\ga=\sum_{x\in C_\ga} x$ has degree zero and lies in the center of the superalgebra $\C[\SC(d)]$, so
it acts on  $V^\la$ as multiplication by a constant.
This constant is the central character $\f_\ga(\la)$
obtained from the formula
\begin{equation}\label{ccf}
\f_\ga(\la) = \frac{|C_\ga|}{\dim V^\la} \,\zeta^\la(C_\ga).
\end{equation}
When $\la\in \SP^-(d)$, the central character  $\f_\ga(\la)$ of $V^\la=V^\la_0\oplus V^\la_1$ equals to the central characters of the irreducible representations $V^\la_0,V^\la_1\in \SC(d)^\wedge_-$.

Now, (\ref{inner-prod1}) and (\ref{ccf}) give a fact central to our subsequent discussions.

\begin{lemma}\label{key}
Let $c_\la$ be as in (\ref{dim}). We have
\begin{equation*}\label{conclusion}
\sum_{\ga\in \OP(d)} \vt_\ga\,\f_\ga(\la)\f_\ga(\mu) =
\left\{
\begin{array}{cl}
\delta_{\la\mu}/2c_\la^2\ \ &\text{if}\ \ \la\in \SP^+(d)
 \\
\delta_{\la\mu}/c_\la^2\ \ &\text{if}\  \ \la\in \SP^-(d)
\end{array}\right.
\end{equation*}
\end{lemma}

\subsection{Center}

Let $C_1,\cdots,C_k$ be conjugacy classes in a finite group $G$ and let
$n(C_1,\cdots,C_k)$ be the number of solutions $(g_1,\cdots,g_k)\in C_1\times\cdots \times C_k$ of the equation
$g_1\cdots g_k=1$. Then we have
\begin{equation}\label{Serre}
n(C_1,\cdots,C_k)
 = |G| \sum_{\la \in G^\wedge} \Big(\frac{\dim V^\la}{|G|}\Big)^2 \prod_i \f_{C_i}(\la),
\end{equation}
where $\f_{C_i}(\la)$ are the central characters of $V^\la$
(cf. Theorem 7.2.1 of \cite{Sr}).

The central element $\ep$ acts on the center $\CZ(\C[\SC(d)])$ of the (ungraded) group algebra $\C[\SC(d)])$
with $\pm 1$ eigenvalues. We denote by
\begin{equation}\label{Hspace}
\CZ_0^+ \ \subset\   \CZ
\end{equation}
the $(-1)$-eigenspace
consisting of even degree elements.
This space has a basis
$$
\left\{u_\ga= \tfrac12(\overline{C}_\ga-\ep\overline{C}_\ga):\ga\in \OP(d)\right\}.
$$
For notational simplicity, we set $\1=(1^d)$. Since
$u_\a u_\b\in \CZ_0^+$ and $u_{\1}u_\a=u_\a$ for all $\a,\b\in\OP(d)$,
the space $\CZ_0^+$ is a commutative associative algebra with identity $u_\1$.

By (c) and (d) in Section 1.3 and  (\ref{Serre}), we obtain:

\begin{lemma}\label{str-cons}
If $u_\a u_\b = \sum_\ga a^\ga_{\a \b} u_\ga$, then the structure constants $a_{\a\b}^\ga$ are given by
\begin{align*}
a^\ga_{\a\b} &=
\vt_\ga\Big(\hspace{-3pt}\sum_{\la\in \SP^+(d)} 2 c_\la^2 \, \f_\a(\la)\f_\b(\la)\f_\ga(\la)
\ + \!\sum_{\la\in \SP^-(d)}  c_\la^2 \, \f_\a(\la)\f_\b(\la)\f_\ga(\la) \Big).
\end{align*}
\end{lemma}

\section{Calculation of genus zero spin Hurwitz numbers}

In this section, we calculate the spin Hurwitz numbers of genus $h=0$, following the arguments of \cite{EOP}.
We  generally follow the notation and terminology in \cite{EOP}.

\subsection{Quadratic form}

Consider a degree $d$  map
\begin{equation}\label{intro-map}
f:C \ \to\   \P^1
\end{equation}
ramified only over fixed points $y^1,\cdots,y^k\in \P^1$ with
ramification profile $\a^i\in \OP(d)$ at $y^i$
satisfying (\ref{RH}).
Let $N=\O(-1)$ and
let $L=N_f$ denote the theta characteristic on $C$ defined by (\ref{TwistedTheta}).
For each (connected) component $C_i$ of $C$ where $1\leq i\leq n$,
the theta characteristic $L_i=L|_{C_i}$ on $C_i$ determines a
quadratic form $q_{L_i}$ on the group  $J_2(C_i)$ of elements of order two in the Jacobian of $C_i$ by
$$
q_{L_i}(\rho_i) \equiv h^0(L_i\otimes \rho_i) + h^0(L_i)\ \ (\text{mod}\ 2)
$$
such that
$$
(-1)^{h^0(L_i)} = 2^{-g(C_i)}\sum_{\rho_i\in J_2(C_i)} (-1)^{q_{L_i}(\rho_i)}
$$
For $\rho=(\rho_1,\cdots,\rho_n)$ in $J_2(C)=J_2(C_1)\times \cdots \times J_2(C_n)$, let
$q_L(\rho) = \underset{i}{\sum} \,q_{L_i}(\rho_i)$. Then
\begin{equation}\label{Arf}
(-1)^{p(f)} = \prod_i (-1)^{h^0(L_i)} = 2^{\frac{\chi(C)}{2}-n} \sum_{\rho\in J_2(C)} (-1)^{q_L(\rho)}
\end{equation}

\subsection{Canonical lift}

Each $\rho\in J_2(C)$ defines an unramified double cover $C_\rho\to C$ which, when composed with  $f$, gives  a degree $2d$ cover
\begin{equation*}\label{double-cover}
f_\rho:C_\rho\ \to\ \P^1.
\end{equation*}
Let $\si$ be  the fixed point free
involution in the symmetric group $\SS(2d)$ given by
the covering transformation permuting the sheets of $C_\rho\to C$.
By our construction, the monodromy group of $f_\rho$  lies in the centralizer of
the involution $\si$ in $\SS(2d)$,
which is the group $\SB(d)$ (see Section 1.1).
The monodromy of $f_\rho$ thus defines a homomorphism
\begin{equation}\label{mono-hom}
M_{f_\rho}:\pi_1(\P^{\st})\ \to\ \SB(d),
\end{equation}
where $\P^{\st}=\P^1\setminus \{y^1,\cdots,y^k\}$.

One can choose a small loop $\delta_i$ encircling only the branch point $y_i$ such that
\begin{itemize}
\item
$\pi_1(\P^{\st}) = \big\lb\, \de_1,\cdots,\de_k \,|\,\prod_i \delta_i\ =\ 1\,\big\rb$,
\item
$M_{f_\rho}(\delta_i)$ is conjugate to a pure permutation $g_i$  of cycle type $\a^i\in\OP(d)$ in $\SB(d)$.
\end{itemize}
Then by (a) in Section~1.2,
$$
\theta^{-1}\big(M_{f_\rho}(\delta_i)\big)\ \subset \ C_{\a^i}\ \sqcup\ \ep C_{\a^i}
$$
where $C_{\a^i}$ is the conjugacy class of the pure permutation $g_i$  in the Sergeev group $\SC(d)$.
The monodromy of $f_\rho$ is said to have a {\em canonical lift to} $\SC(d)$ if there exists
a homomorphism $\widehat{M}_{f_\rho}:\pi_1(\P^{\st})\to \SC(d)$
such that $\widehat{M}_{f_\rho}(\delta_i)\in C_{\a^i}$ for all $i$ and the diagram commute:
$$
\xymatrix{
 & \SC(d)   \ar[d]^\theta  \\
 \pi_1(\P^{\st})  \ar[ur]^{\widehat{M}_{f_\rho}} \ar[r]_{\ \ M_{f_\rho}} & \SB(d)  }
$$

The following fact is a special case of Theorem~1 of \cite{EOP}: the case of $\P^1$.

\begin{prop}\label{EOP-thm1}
$q_{L}(\rho)=0$ if and only if the monodromy of $f_\rho$ has a canonical lift to $\SC(d)$.
\end{prop}

\subsection{Weighted count}

Let $G$ be $\SS(d), \SB(d)$ or $\SC(d)$, and let $C_{\a^i}$ denote the conjugacy class
of a pure permutation in $G$ with cycle type $\a^i\in \OP(d)$.
We set
\begin{equation}\label{data}
M = \{\a^1,\cdots,\a^k\}
\end{equation}
and denote by
$H_G(M)$
the set of homomorphisms $\psi:\pi_1(\P^{\st})\to G$ sending the conjugacy class of  the loop $\delta_i$
into the conjugacy class $C_{\a^i}$.
Taking into account the action of $G$ by conjugation, we set
$$
h_G(M) = \frac{|H_G(M)|}{|G|}
$$

The groups $\SB(d)$ and $\SC(d)$ have natural homomorphisms to $\SS(d)$ by definition.
Given a homomorphism $\phi\in H_{\SS(d)}(M)$,
let $H_G(M;\phi)$ be the set of homomorphisms $\psi\in H_G(M)$ with commutative diagram
$$
\xymatrix{
 & G   \ar[d] \\
 \pi_1(\P^{\st})  \ar[ur]^{\psi} \ar[r]_{\ \ \phi} & \SS(d)  }
$$
where $G\to \SS(d)$ is the natural homomorphism. The weighted count of such homomorphisms is
$$
h_G(M;\phi) = \frac{|H_G(M;\phi)|}{|G|}
$$
By (\ref{Serre}) and definition, $h_{G}(M;\phi)$ and $h_G(M)$ satisfy
\begin{equation}\label{h-cha}
\sum_{\phi\in H_{\SS(d)}(M)}  h_{G}(M;\phi) = h_G(M)
 =
\sum_{\la\in G^\wedge} \Big(\frac{\dim V^\la}{|G|}\Big)^2 \prod_i \f_{C_i}(\la)
\end{equation}

\begin{rem}\label{Hurwitz}
There is a bijection between ramified covers $f:C\to \P^1$ (as in (\ref{intro-map}))
and orbits of the action of $\SS(d)$ on $H_{\SS(d)}(M)$ by conjugation. This bijection  is given by
the monodromy $M_f:\P^{\st}\to \SS(d)$ of the map $f$.
The order of the stabilizer of $M_f$ is
$|\Aut(f)|$  and hence
\begin{equation}\label{Hurwitz}
h_{\SS(d)}(M) = \frac{1}{|\SS(d)|}\sum_{O_f} |O_f| = \sum_f \frac{1}{|\Aut(f)|}
\end{equation}
where $O_f$ is the orbit of $M_f$.
This is the ordinary Hurwitz number that counts ramified covers of $\P^1$ with ramification data specified by $M$
in (\ref{data}).
\end{rem}

\begin{lemma}\label{lemma3}
Let $f:C\to \P^1$ and $O_f$ be as in Remark~\ref{Hurwitz} and let $\phi\in O_f$.
If the domain $C$ has $n$ (connected) components
$C_1,\cdots,C_n$, then we have
$$
|J_2(C)| = 2^{-d+n} |H_{\SB(d)}(M;\phi)|.
$$
\end{lemma}

\begin{proof}
The proof  is identical to that of Lemma~3 in \cite{EOP}.
Assigning to each $\rho\in J_2(C)$ the monodromy $M_{f_\rho}$ given in (\ref{mono-hom})
defines a bijection between $J_2(C)$ and  orbits of the action of $\Z_2^d\subset \SB(d)$ on
$H_{\SB(d)}(M;\phi)$ by conjugation.
Let $\rho=(\rho_1,\cdots,\rho_n)$  and $C_{\rho_i}\to C$ be the double cover determined by $\rho_i$ in $J_2(C_i)$.
Then the stabilizer of $M_{f_\rho}$
is generated by $\si_1,\cdots,\si_n$ where  $\si_i$ is the involution permuting the sheets of
$C_{\rho_i}\to C_i$. So, every orbit of the action of $\Z_2^d$ on $H_{\SB(d)}(M;\phi)$ has $2^{d-n}$ elements and hence
$$
|H_{\SB(d)}(M;\phi)| = \sum_{\text{orbits}} 2^{d-n} =
 \sum_{\rho\in J_2(C)} 2^{d-n}
$$
This completes the proof of the lemma.
\end{proof}

\subsection{Proof of (\ref{h=0})}

Let $\phi\in H_{\SS(d)}(M)$ be as in Lemma~\ref{lemma3}.
By our choice of the conjugacy classes $C_{\a^i}$ in $G$,
the homomorphism $\theta:\SC(d)\to \SB(d)$ induces an one-to-one function
\begin{equation}\label{1-1}
H_{\SC(d)}(M;\phi) \ \to\   H_{\SB(d)}(M;\phi).
\end{equation}
Moreover, $q_L(\rho)=0$ if and only if the monodromy
$M_{f_\rho}$ lies in the image of (\ref{1-1}) by Proposition~\ref{EOP-thm1}. Thus by
$|\SC(d)|=2|\SB(d)|=2^{d+1}d!$, (\ref{Arf}) and Lemma~\ref{lemma3}, we have
\begin{align}\label{step1}
(-1)^{p(f)}\ &=
2^{\frac{\chi(C)}{2}-n} \sum_{\rho\in J_2(C)} (-1)^{q_L(\rho)} =
2^{\frac{\chi(C)}{2}-d}\left(2\left| H_{\SC(d)}(M;\phi) \right| -
\left|H_{\SB(d)}(M;\phi)\right|\right) \notag \\
&= 2^{\frac{\chi(C)}{2}}d!\left[ 4h_{\SC(d)}(M;\phi) - h_{\SB(d)}(M;\phi)\right].
\end{align}
Now, it follows that
\begin{align*}
H^{0,+}_{\a^1,\cdots,\a^k}
&=
\sum_{f} \frac{(-1)^{p(f)}}{|\Aut(f)|}
 =
\sum_{\phi\in H_{\SS(d)}(M)} 2^{\frac{\chi(C)}{2}}\left[ 4h_{\SC(d)}(M;\phi) - h_{\SB(d)}(M;\phi)\right] \\
&=
2^{\frac{\chi(C)}{2}}\sum_{\la\in C(d)^\wedge_-}
2^2\Big(\frac{\dim V^\la}{|\SC(d)|}\Big)^2 \prod_i \f_{\a^i}(\la)\\
&=
2^{\frac{\chi(C)+2}{2}}
\Big(
 \sum_{\la\in \SP^+(d)} 2c_\la^{2} \prod_i \f_{\a^i}(\la)
 +\sum_{\la\in \SP^-(d)} c_\la^{2}\prod_i \f_{\a^i}(\la)
\Big),
\end{align*}
where the second equality follows from (\ref{Hurwitz}) and (\ref{step1}),
the third from (\ref{decomp-irr}) and (\ref{h-cha}), and the last from
(c) and (d) in Section~1.3. This completes the proof of (\ref{h=0}).

\section{A proof of Gunningham's formula (\ref{G}) }

For $\a^1,\cdots,\a^k\in\OP(d)$,   we set
\begin{equation*}\label{def}
H(h,p)_{\a^1,\cdots,\a^k} =
2^{-\frac{\chi(C)+\chi(D)}{2}}\, H^{h,p}_{\a^1,\cdots,\a^k},
\end{equation*}
where $\chi(D)=2-2h$ and $\chi(C)$ is the domain Euler characteristic.

Using the Einstein summation convention,
we raise indices by the formula
$$
H(h,p)_{\a^1,\cdots,\a^s}^{\b^1,\cdots,\b^r} =
\vt_{\b^1}\cdots\vt_{\b^r}H(h,p)_{\a^1,\cdots,\a^s,\b^1,\cdots,\b^s}
$$
For notational convenience, we denote multi-indices $\a^1,\cdots,\a^s$ by boldface index $\ba$.
Then the gluing theorem (Theorem~\ref{T:LP2}) can be written as:
\begin{equation}\label{gluing}
\begin{array}{c}
H(h_1+h_2,p_1+p_2)_{\ba,\be}^{\bb,\bd} =
H(h_1,p_1)_{\ba}^{\bb,\ga}H(h_2,p_2)^{\bd}_{\be,\ga},
\\
H(h+1,p)_{\ba}^{\bb} =
H(h,p)_{\ba,\ga}^{\bb,\ga}. {\displaystyle \phantom{\int_A^{B}}\, }
\end{array}
\end{equation}

\begin{rem}
By the additivity of Euler characteristic,
if
$\chi(C)$, $\chi(C_1)$ and $\chi(C_2)$ are the domain Euler characteristics
for
$H(h_1+h_2,p_1+p_2)_{\ba,\be}^{\bb,\bd}$, $H(h_1,p_1)_{\ba}^{\bb,\ga}$ and
$H(h_2,p_2)^{\bd}_{\be,\ga}$, then
$$
\chi(C) = \chi(C_1)\ +\ \chi(C_2)\ -\ 2\ell(\ga).
$$
The first formula in (\ref{gluing}) thus follows
from Theorem~\ref{T:LP2} and the definition $\vt_\ga=2^{\ell(\ga)+1}z_\ga$.
The proof of the second formula is the same.
\end{rem}

Now, observe that
\begin{align*}
&H(1,+)_{\a^1,\cdots,\a^k} = H(0,+)_{\a^1,\cdots,\a^k,\ga}^\ga \Spacing \\
&=
\sum_{\ga}\vt_\ga
\Big(
\sum_{\la\in \SP^+(d)} 2c_\la^2\, \prod_i\f_{\a^i}(\la)\f_\ga(\la)\f_\ga(\la)
+\sum_{\la\in \SP^-(d)} c_\la^2\, \prod_i\f_{\a^i}(\la)\f_\ga(\la)\f_\ga(\la)
\Big)
\Spacing \\
&=
\sum_{\la\in \SP^+(d)} \prod_i\f_{\a^i}(\la)
+\sum_{\la\in \SP^-(d)} \prod_i\f_{\a^i}(\la),
\end{align*}
where the first equality follows from the gluing theorem, the second from
(\ref{h=0}), the last from Lemma~\ref{key}.
In this way,
the gluing theorem, (\ref{EOP}) and (\ref{h=0}) inductively give
\begin{equation}\label{another-G}
H(h,\pm)_{\a^1,\cdots,\a^k} =
\sum_{\la\in \SP^+(d)}2^{1-h} c_\la^{2-2h}\prod_i\f_{\a^i}(\la)
\pm\sum_{\la\in \SP^-(d)} c_\la^{2-2h}\prod_i\f_{\a^i}(\la).
\end{equation}
This completes the proof of  (\ref{G}).

\begin{rem}\label{GGluing}
By Lemma~\ref{key}, one can easily obtain  (\ref{gluing}) from  (\ref{another-G}).
Therefore, the Gunningham's formula (\ref{G}) implies the gluing theorem (Theorem~\ref{T:LP2}).
\end{rem}

\section{A TQFT formalism}

This section discusses  a TQFT formalism via the gluing theorem.
Our approach is analogous to the that in \cite{BP2}. We refer to \cite{A,K} for Frobenius algebra and TQFT.

\subsection{Extended gluing theorem}

Recall that spin Hurwitz numbers are not defined for $(h,p)= (0,-)$ because
the only theta characteristic on $\P^1$ is the even theta characteristic $\O(-1)$.
In lieu of Lemma~\ref{key} and the formula (\ref{another-G}), if we define
\begin{equation}\label{extention}
H(0,-)_{\a^1,\cdots,\a^k} =
\sum_{\la\in \SP^+(d)} 2c_\la^{2}\prod_i\f_{\a^i}(\la)
-\sum_{\la\in \SP^-(d)} c_\la^{2}\prod_i\f_{\a^i}(\la),
\end{equation}
then the gluing theorem (\ref{gluing}) extends to include the case  $(h,p)=(0,-)$.

\subsection{Functor}

The (extended) gluing theorem naturally induces a functor between tensor categories,
$$
\SH:\,2\Cob^\pm\ \to\ \Vect.
$$
Here $\Vect$ denotes the usual tensor category of complex vector spaces.
The objects of the category $2\Cob^\pm$ are finite unions of oriented circles.
The morphisms are given by pairs $(D,p)$, where $D$ is an  oriented cobordism (modulo diffeomorphism relative to the boundary) between two objects and $p\in \Z_2$.
We denote by
$$
D_r^s(h,p)
$$
the connected genus $h$ with parity $p$ cobordism  from a disjoint union of $r$
circles to a disjoint union of $s$ circles.
The composition of morphisms, obtained by concatenation of cobordisms, respects the $\Z_2$-grading (or parity).
The tensor structure on the category $2\Cob^\pm$ is given by disjoint union.

We define $\SH(S^1)=\F$ to be the vector space with basis $\{v_\a:\a\in\OP(d)\}$ labelled by odd partitions $\a\in \OP(d)$ and let
$$
\SH\left(S^1\coprod \cdots \coprod S^1\right) = \F\otimes\cdots \otimes\F.
$$
For connected  cobordisms $D_r^s(h,p)$, we define a linear map
$$
\SH\left(D_r^s(h,p)\right) : \F^{\otimes r}\  \to\  \F^{\otimes s}
\ \ \ \ \ \text{by}\ \ \ \ \ v_{\ba} \ \mapsto\ H(h,p)_{\ba}^{\bb}v_{\bb},
$$
where $v_{\ba}=v_{\a^1}\otimes \cdots \otimes v_{\a^r}$ for  $\ba=\a^1,\cdots,\a^r$.
Taking tensor product, we extend this definition to disconnected cobordisms.

$\SH$ takes the identity morphism $D_1^1(0,+)$  to the identity map on $\F$,
$$
v_\a \ \mapsto\   H(0,+)_\a^\b v_\b   =  v_\a,
$$
by Lemma~\ref{L:TQFT}\,(b) below.
$\SH$ also takes, by the (extended) gluing theorem,
the concatenation of cobordisms to the composition of linear maps
(cf. Proposition 4.1 of \cite{BP1}).
Therefore, $\SH$ is a well-defined functor.
In particular, one obtain a 2d (two dimensional)  TQFT, $2\Cob^+\to \Vect$, by restricting to even cobordisms.

Recall that $\1$ denotes the partition $(1^d)\in \SP(d)$. 
The fact below follows from the same calculation of Hurwitz numbers because
the parity of the maps with domain  $\P^1$ is even.

\begin{lemma}\label{L:TQFT}
For $\a,\b\in\OP(d)$, we have:
\begin{itemize}
\item[(a)]
$H(0,+)^\a  =  \delta_{\1\a}v_\1$.
\item[(b)]
$H(0,+)_\a^\b  = \delta_{\a\b}$.
\end{itemize}
\end{lemma}

\subsection{Frobenius algebra}

The even cap $D^1(0,+)$ defines a unit $U:\C\to \F$ by
$U(1)=H(0,+)^\a v_\a =v_\1$
(see Lemma~\ref{L:TQFT}\,(a)), while
the even pair of pants $D_2^1(0,+)$ defines a multiplication $\F\otimes\F \to  \F$ by
$$
v_\a\otimes v_\b\ \mapsto\ v_\a v_\b  = H(0,+)_{\a\b}^\ga v_\ga.
$$
The algebra $\F$ is thus, by Lemma~\ref{str-cons} and (\ref{another-G}), isomorphic to
the algebra $\CZ_0^+$ in (\ref{Hspace}) with isomorphism $v_\a\mapsto u_\a$.
The Frobenius algebra structure on  $\F=\CZ_0^+$ is given by the  counit
$$
T:\F \  \to\  \C\ \ \ \text{where}\ \ \
T(v_\a) = \SH\left(D_1(0,+)\right)(v_\a) =   H(0,+)_\a  =  \delta_{\1\a}/\vt_\1.
$$

Using Lemma~\ref{key} and the structure constants $a_{\a\b}^\ga=H(0,+)_{\a\b}^\ga$,
one can find, by hand, an idempotent basis
$\left\{e_\la:\la\in \SP(d)\right\}$ ($e_\la e_\mu=\delta_{\la\mu}e_\la$):
$$
e_\la =
\left\{
\begin{array}{cl}
{\displaystyle
\sum_{\a\in\OP(d)} 2c_\la^2\, \vt_\a \,f_\a(\la)\,v_\a }
\ \ &\text{if}\ \ \la\in\SP^+(d) \\
{\displaystyle
 \sum_{\a\in\OP(d)} c_\la^2\, \vt_\a \,f_\a(\la)\,v_\a  }
\ \ &\text{if}\ \ \la\in\SP^-(d)
\end{array}
\right.
$$
The Frobenius algebra $\F=\CZ_0^+$ is semisimple since  $\F=\oplus \C e_\la$ where
$\C e_\la$'s are  one dimensional Frobenius algebras with counit $e_\la\mapsto t_\la=T(e_\la)$.
Since $\f_\1(\la)=1$,
$$
t_\la = \left\{
\begin{array}{cl}
2c_\la^2\ &\text{if}\ \ \la\in \SP^+(d) \\
c_\la^2\  &\text{if}\ \ \la\in \SP^-(d)
\end{array}
\right.
$$
Observe that the Frobenius algebra $\F$ has an involution given by
\begin{equation}\label{involution}
A := \SH\left(D_1^1(0,-)\right):\F\ \to\ \F.
\end{equation}

\begin{rem}
In \cite{G}, Gunningham constructed a fully extended 2d spin TQFT, which is a functor from the 2-category of spin cobordisms  to the category of superalgebras. The spin TQFT gives the formula (\ref{G}) and hence the gluing theorem (\ref{gluing}). In our case, on the other hand, we obtained from the gluing theorem a modified 2d TQFT, which includes both odd and even spin Hurwitz numbers and whose underlying Frobenius algebra has an additional structure, the involution (\ref{involution}) induced from (\ref{extention}). To see the full spin TQFT, one may need more information than the underlying Frobenius algebra with an involution.

\end{rem}

\subsection{Dimension zero GW invariants of K\"{a}hler surfaces}

The semisimplicity, $\F=\oplus \C e_\la$, implies that
 $e_\la$ is an eigenvector with eigenvalue $t_\la^{-1}$ for the genus adding operator
$$
G = \SH\left(D_1^1(1,+)\right):\F\ \to\ \F.
$$
One can also see, by simple calculation, that  $e_\la$ is an eigenvector with eigenvalue $(-1)^{\ell(\la)}$ for the involution $A$ of $\F$.

Now, noting that $U(1)=v_\1=\sum e_\la$, $H(h,p) = \SH\left(D(h,p)\right)(1)$ and
$$
\SH\left(D(h,p)\right) =
\SH\left(D_1(0,+)\circ D_1^1(h,p)\circ D^1(0,+)\right)
 = T\circ A^p\circ G^h\circ U,
$$
we can write the dimension zero GW invariants of K\"{a}hler surfaces (\ref{dim=0GW})
succinctly as
$$
GT_d^{h,p} = 2^{\frac{\chi(C)+\chi(D)}{2}}H(h,p) =
2^{\frac{\chi(C)+\chi(D)}{2}}
\Big( \sum_{\la\in\SP(d)} (-1)^{p\cdot\ell(\la)}\, t_\la^{\frac{\chi(D)}{2}} \Big),
$$
where $\chi(D)=2-2h$ is the Euler characteristic of the smooth canonical divisor and $\chi(C)$ is the domain Euler characteristic.

\bigskip

\noindent {\em  Department of  Mathematics,  University of Central Florida, Orlando, FL 32816}

\medskip

\noindent {\em e-mail:}\ \ {\ttfamily junho.lee@ucf.edu}

\end{document}